\documentclass[11pt,letterpaper]{amsart}

\usepackage{geometry}
\usepackage[english]{babel}
\usepackage{amsmath,amssymb,mathtools,enumerate,latexsym,graphicx,cite}
\usepackage{xcolor}
\usepackage{hyperref}
\hypersetup{
  colorlinks=true,
  linkcolor=black,
  citecolor=black,
  urlcolor=black,
  pdftitle={On the stability of rank-one factors for Anosov Zk-actions},
  pdfauthor={Danijela Damjanovi\'c and Jiesong Zhang}
}

\newif\ifshowchanges
\showchangestrue

\newenvironment{proof*}[1]{\medskip\noindent\textbf{#1\ }}{\hspace*{\fill}$\Box$\medskip}
\newtheorem{theorem}{Theorem}[section]
\newtheorem{lemma}[theorem]{Lemma}
\newtheorem{proposition}[theorem]{Proposition}

\newtheorem{definition}[theorem]{Definition}
\theoremstyle{remark}
\newtheorem{remark}[theorem]{Remark}

\newtheorem{question}{Question}

\newtheorem{mainthm}{\bfseries Theorem}

\newcommand{\E}{{\mathbf E}}

\newcommand{\W}{{\mathbf W}}

\newcommand{\ZZ}{{\mathbb Z}}

\newcommand{\SL}{{\rm SL}}

\newcommand{\be}[1]{\begin{equation} \label{#1} }
\newcommand{\ee}{\end{equation}}
\newcommand{\beq}{\begin{equation}}

\def \Diff{{\rm Diff}}

\def \W{{\mathcal W}}

\def \W{\mathcal{W}}

\def \hx0{\hat{x_0}}

\def \E{\mathcal{E}}

\def \L{\mathcal{L}}

\author{Danijela Damjanovi\'c}
\address{Department of Mathematics, Kungliga Tekniska h\"ogskolan, Lindstedtsv\"agen 25, SE-100 44 Stockholm, Sweden}
\email{ddam@kth.se}
\author{Jiesong Zhang}
\address{Department of Mathematics, Kungliga Tekniska h\"ogskolan, Lindstedtsv\"agen 25, SE-100 44 Stockholm, Sweden}
\email{jiesongz@kth.se}
\title[STABILITY OF RANK-ONE FACTORS]{On the stability of rank-one factors for Anosov $\mathbb Z^k$-actions}

\begin{document}
\begin{abstract}
We prove $C^1$-stability of $C^\infty$ rank-one factors for a large class of higher rank Anosov $\ZZ^k$-actions. If the factor is a projection of a smooth fiber bundle and the action contains an element which acts as the identity on the rank-one factor and is Anosov on every fiber, then any sufficiently $C^1$-small $C^\infty$ perturbation of the action will also have a $C^\infty$ rank-one factor. In particular, we prove the $C^1$-stability of the product structure for Anosov $\ZZ^2$-actions which are obtained as a (diagonal) product of two $\ZZ$-Anosov actions. 
\end{abstract}

\maketitle

\tableofcontents

\section{Introduction}
A diffeomorphism $f\in \Diff^\infty(M)$ is Anosov if there is a continuous $Df$-invariant splitting $TM=E^s\oplus E^u$ such that vectors in $E^s$ (resp. $E^u$) are exponentially contracted (resp. expanded) at a uniform rate. Anosov diffeomorphisms have been extensively studied. The main models for Anosov systems are the {\it algebraic} ones: automorphisms of nilmanifolds whose derivatives have all eigenvalues off the unit circle.
In \cite{Franks, Manning}, it is proved that an Anosov diffeomorphism $f$ on a nilmanifold is necessarily topologically conjugate to an algebraic model, which is then called a {\it linearization} of $f$.
The topological conjugacy, however, need not have higher regularity.
Nevertheless, if an Anosov diffeomorphism on a nilmanifold has many smooth symmetries, i.e. it is contained in a smooth $\mathbb Z^k$-action with $k\ge 2$, then the regularity of the conjugacy can be improved to $C^\infty$, under the condition that the linearization of the  $\mathbb Z^k$-action has no (algebraic) {\it rank-one factors}. This was first proven locally, in the neighborhood of algebraic models \cite{ks94}, and much later  globally \cite{fks13, hw14}.  

In general, a $\mathbb Z^k$-action $\alpha: \mathbb Z^k\to \Diff^\infty(M)$ on a compact smooth manifold $M$ is said to be Anosov if it contains an Anosov diffeomorphism.
If the action $\alpha$ factors via a smooth projection map to a smooth $\mathbb Z$-action (possibly on a different manifold), then we say that the action has a smooth rank-one factor.

In the 1990s, Katok and Spatzier conjectured that Anosov $\mathbb Z^k$-actions with no smooth rank-one factors are {\it globally rigid}: they should be smoothly conjugate to an {\it algebraic} Anosov $\mathbb Z^k$-action on a nilmanifold, after passing to a finite cover and a finite-index subgroup if necessary \cite{ks94}.
This conjecture for general manifolds is still open (see \cite{DX} for most recent developments and some history).

If the above conjecture is true, then the class of Anosov actions with no rank-one factors is open. We consider here the complementary class of Anosov actions {\it that have rank-one factors}.
We ask the following question: 

\begin{question} Is the class of Anosov $\mathbb Z^k$-actions, $k\ge 2$, with a smooth rank-one factor open in the $C^1$ topology?
\end{question}

An affirmative answer to the question above, together with the Katok-Spatzier conjecture,  gives a better topological picture of the space of smooth $\mathbb Z^k$ ($k\ge 2$) Anosov actions. 

In this note, we show that for a certain large subclass of Anosov $\mathbb Z^k$-actions, the property of having a smooth rank-one factor indeed persists under $C^1$-small smooth perturbations.

This should be contrasted with recent findings \cite{vin25} for $\mathbb R^k$-actions with $k\ge 2$. An $\mathbb R^k$-action is Anosov if there is an element of the action that is normally hyperbolic with respect to the $\mathbb R^k$-orbit foliation.
In \cite{vin25}, Vinhage demonstrates {\it instability} of rank-one factors for Anosov $\mathbb R^k$-actions by constructing a time change of a product $\mathbb R^2$-action that eliminates all smooth rank-one factors.

The main result of this paper (Theorem \ref{theorem rank 1}) shows {\it stability} of rank-one factors for a large class of Anosov $\mathbb Z^k$-actions. The class of actions we consider is natural: we assume that the factor is a projection of a smooth fiber bundle and that the action contains an element which acts as the identity on the rank-one factor and is Anosov on every fiber. In particular, (diagonal) products of Anosov $\mathbb Z$-actions are of this kind (see Theorem \ref{theorem product}).

We note that the rank-one factors of $\mathbb Z^k$-actions that we consider here are {\it Anosov}, whereas earlier results studied the persistence of identity factors or translation factors for $\mathbb Z^k$-actions that are {\it partially hyperbolic} (for example,  \cite{FD}, \cite{dwx21}, \cite{DSWX}, \cite{LSV}). The mechanisms for persistence of factors in the results on partially hyperbolic actions transverse to a non-Anosov factor, are very different than what we do in this paper. \\

An interesting question is to what extent rigidity theory for Anosov actions with no rank-one factors can be extended to actions with rank-one factors. For example, given an algebraic Anosov $\mathbb Z^k$-action with a non-trivial maximal smooth rank-one factor, is it true that every sufficiently small smooth perturbation is conjugate to the unperturbed action via a conjugacy that is smooth {\it transversely} to the factor?
This follows, even globally on nilmanifolds, from the forthcoming work \cite{DSWX}, but only under certain very strong (exponential) mixing conditions transverse to the factor.
In particular, the conditions in \cite{DSWX} never hold for an action which is a (diagonal) product of two Anosov $\mathbb Z$-actions. 

Besides (diagonal) product actions, we illustrate our result by applying it to Anosov actions on nilmanifolds, see Section \ref{ex:borel-smale} for an example. 

Lastly, we discuss an example with a rank-one factor where neither our results, nor other existing results apply (see Section \ref{non-example}). Both the stability of the rank-one factor and the transverse regularity of conjugacy for perturbations, remain open problems.

\color{black}

\section{Formulation of the results}
Let $\alpha : \mathbb{R}^k \times \mathbb{Z}^\ell \curvearrowright M$ be a locally free action by $C^\infty$ diffeomorphisms. An element $a \in \mathbb{R}^k \times \mathbb{Z}^\ell$ is called \textit{Anosov} if there exists an $a$-invariant dominated splitting, 
\[
TM = E^s_a \oplus T\mathcal{O} \oplus E^u_a,
\]
such that $D\alpha(a)|_{E^s_a}$ is uniformly contracting and  $D\alpha(a)|_{E^u_a}$ is uniformly expanding. Here $\mathcal{O}$ denotes the $\mathbb{R}^k$-orbit foliation (when $k=0$, this term is understood to be the 
zero distribution). The action $\alpha$ is called \emph{Anosov} if it contains an Anosov element.
\begin{definition}[Rank-one factor]
Let $r \geq 1$, and let $\alpha : \mathbb{R}^k \times \mathbb{Z}^{\ell} \curvearrowright M$ be a $C^r$ locally free action. A \emph{$C^r$ rank-one factor} of $\alpha$ consists of the following data:
\begin{itemize}
\item a $C^r$ manifold $N$ with $\dim N \ge 1$, 
\item either a $C^r$, fixed-point-free flow $\psi_t : N \to N$, or a $C^r$ diffeomorphism $f : N \to N$,
\item a $C^r$  submersion $\pi : M \to N$, and
\item a surjective homomorphism $\sigma : \mathbb{R}^k \times \mathbb{Z}^{\ell} \to \mathbb{R}$ (respectively $\sigma : \mathbb{R}^k \times \mathbb{Z}^{\ell} \to \mathbb{Z}$),
\end{itemize}
such that for all $\mathbf{n} \in \mathbb{R}^k \times \mathbb{Z}^{\ell}$ and $x \in M$,
\[
\pi(\alpha(\mathbf{n})x) = \psi_{\sigma(\mathbf{n})}(\pi(x))
\quad \text{or} \quad
\pi(\alpha(\mathbf{n})x) = f^{\sigma(\mathbf{n})}(\pi(x)).
\]
We also allow passing to a finite-index subgroup of $\mathbb{R}^k \times \mathbb{Z}^{\ell}$ or a finite cover of $M$ to which the action lifts.

\end{definition}

\subsection{Fiberwise Anosov rank-one factors}
Motivated by the notion of fiberwise Anosov diffeomorphisms
\cite{fg16,zha24}, we introduce a class of $C^\infty$ rank-one factors, called \textit{fiberwise Anosov rank-one factors}, and show that this class is stable under $C^1$ perturbations.

\begin{definition}
A diffeomorphism $f \in \Diff^\infty(M)$ of a closed manifold $M$ is called \textit{fiberwise Anosov over the identity (FAOI)} if there exists a $C^\infty$ fiber bundle $\pi : M \to N$ with closed, connected fibers such that 
\[
\pi \circ f = \pi,
\]
and the restriction of $f$ to each fiber $\pi^{-1}(y)$ is Anosov.

For a $C^\infty$ Anosov action $\alpha : \mathbb{Z}^k \curvearrowright M,k\geq 2$ with a  $C^\infty$ rank one factor given by $\pi : M \to N$, if $\alpha(\textbf{n})$ is FAOI for some $\textbf{n} \in \ZZ^k$ with respect to $\pi : M \to N$, then the rank-one factor given by $\pi$ is called a \textit{fiberwise Anosov rank-one factor}. 
\end{definition}
It follows from the definition that a fiberwise Anosov rank-one factor is always a $C^\infty$ rank-one factor. Our main result states that fiberwise Anosov rank-one factors are $C^1$-stable. In particular, they give a class of $C^\infty$ rank-one factors that are $C^1$-stable. 
\begin{mainthm}\label{theorem rank 1}
Let $\alpha_0 : \mathbb{Z}^k \curvearrowright M$, $k\ge 2$ be an Anosov action with a fiberwise Anosov rank-one factor $\pi_0$. Then any $C^\infty$ action $\alpha : \mathbb{Z}^k \curvearrowright M$ that is sufficiently $C^1$-close to $\alpha_0$ has a fiberwise Anosov rank-one factor $\pi$. Moreover, the conjugacy $h$ provided by structural stability maps the fibers of $\pi_0$ to the fibers of $\pi$. 
\end{mainthm}

In particular, if the rank-one factors of an Anosov $\ZZ^k$-action are all fiberwise Anosov, then for the perturbations of this action, the topological rank-one factors and smooth rank-one factors coincide.

The $C^1$-closeness of $\alpha$ to $\alpha_0$ is used in two ways. First, since we do not assume that $M$ is an infranilmanifold, we invoke the stability of Anosov diffeomorphisms to obtain a H\"older conjugacy (see Lemma~\ref{lemma topo conjugacy}). Second, it is used to establish partial hyperbolicity and dynamical coherence for the perturbation of an FAOI diffeomorphism (see Lemma~\ref{lemma FAOI ph}). 

\begin{remark}\label{rem finite regularity}
Although our focus is on the $C^\infty$ category, Theorem~\ref{theorem rank 1} admits a finite-regularity version. Suppose that $\alpha_0$ is $C^r$ for some $r >2$, and that the given fiberwise Anosov factor has a $C^{r-1}$ fiber bundle projection. Then sufficiently $C^1$-small $C^r$ perturbations admit a $C^s$ fiberwise Anosov rank-one factor  for every $1<s<r-1$, with fiberwise Anosov understood in the corresponding category. 
\end{remark}


\subsection{Rank-one factors of product actions}
A natural class of Anosov actions with fiberwise Anosov rank-one factors is given by product actions.
\begin{definition}[Product action]
Let $M = M_1 \times M_2$ be a product of two closed manifolds. A $C^\infty$ Anosov action $\alpha: \mathbb{Z}^2 \curvearrowright M$ (or $\alpha: \mathbb{R}^2 \curvearrowright M$) is called a \emph{product action} if 
\[
\alpha(m,n)=f_1^m \times f_2^n 
\quad \text{or} \quad
\alpha(t,s)=\psi_1(t) \times \psi_2(s),
\]
where each $f_i$ (respectively $\psi_i$) is an Anosov diffeomorphism (respectively flow) on $M_i$. 
\end{definition}
In \cite{vin25}, Vinhage performs a time change along the flow direction of a product Anosov $\mathbb{R}^2$-action such that the perturbed action has no $C^\infty$ rank-one factor. For $\ZZ^2$-actions, such a construction is more subtle, as there is no center direction. In contrast, we prove the following stability result: 
\begin{mainthm}\label{theorem product} 
Let $\alpha_0 : \mathbb{Z}^2 \curvearrowright M$ be a product Anosov action on a closed manifold $M=M_1\times M_2$. Then any $C^\infty$ action $\alpha : \mathbb{Z}^2 \curvearrowright M$ that is sufficiently $C^1$-close to $\alpha_0$ is smoothly conjugate to a product Anosov action 
\[\alpha(m,n)=(f_1')^m\times(f_2')^n\]
on $M_1\times M_2$. 
\end{mainthm}

We note that although $\alpha$ is smoothly conjugate to a product Anosov action, the components $f'_i$ are in general only topologically conjugate to $f_i$.



\subsection{Example with a rank-one factor that is not a fiberwise Anosov rank-one factor} \label{non-example}Although a fiberwise Anosov rank-one factor is always a $C^\infty$ rank-one factor, the converse does not hold: a $C^\infty$ rank-one factor need not be fiberwise Anosov. For example, let 
\[
A = \begin{pmatrix}
2 & 1 \\
1 & 1
\end{pmatrix} \in \SL(2,\mathbb{Z}).
\]
Consider the $\mathbb{Z}^2$-action on $\mathbb{T}^4$ generated by
\begin{equation}\label{eq para}
f_1 =
\begin{pmatrix}
A & 0 \\
0 & A
\end{pmatrix},
\qquad
f_2 =
\begin{pmatrix}
I & A \\
0 & I
\end{pmatrix}.
\end{equation}
Writing $\alpha(m,n)=f_1^m f_2^n$, the projection $\pi(x,y)=y$ is a $C^\infty$ rank-one factor with quotient map $A$ and $\sigma(m,n)=m$. However, no element of $\ker\sigma$ is Anosov on these fibers: each $f_2^n$ restricts to a translation $x\mapsto x+nAy$. Thus this factor is not fiberwise Anosov.

\begin{question}\label{question shear stability}
Is the smooth rank-one factor $\pi(x,y)=y$ of the action in \eqref{eq para} $C^1$-stable among $C^\infty$ actions?
\end{question}




\subsection*{Acknowledgment.}
We would like to thank Disheng Xu for carefully reading the first version of the manuscript and for his helpful suggestions. J. Z. would also like to thank Jonathan DeWitt for some earlier discussions on related topics. 

\subsection*{AI disclaimer.}
No AI tools have been used for obtaining the results of this paper, nor for the presentation of the results.

\section{Preliminaries}
\subsection{Partially hyperbolic systems}
Recall that a diffeomorphism $f \colon M \to M$ is called~\emph{partially hyperbolic} if the tangent bundle admits a continuous $Df$-invariant splitting
\[
TM = E^s \oplus E^c \oplus E^u,
\]
and there exists an integer $k \geq 1$ such that for every $x \in M$,
\[
\|Df^k(x)|_{E^s}\| < 1 < \|(Df^k(x)|_{E^u})^{-1}\|^{-1},
\]
and
\[
\|Df^k(x)|_{E^s}\| < \|(Df^k(x)|_{E^c})^{-1}\|^{-1}, 
\qquad 
\|Df^k(x)|_{E^c}\| < \|(Df^k(x)|_{E^u})^{-1}\|^{-1}.
\]
It is well known that the invariant distributions $E^s$ and $E^u$ are uniquely integrable into H\"older continuous foliations $\mathcal{W}^s$ and $\mathcal{W}^u$ respectively. The diffeomorphism $f$ is called~\textit{dynamically coherent} if, moreover, there exist $f$-invariant foliations $\mathcal{W}^{cs}$ and $\mathcal{W}^{cu}$ tangent to $E^{cs} := E^s \oplus E^c$ and $E^{cu} := E^c \oplus E^u$ respectively. 

\subsection{Center Lyapunov exponents}
Let $f:M \to M$ be a partially hyperbolic diffeomorphism with dominated splitting $TM =E^s\oplus E^c\oplus E^u$. By Oseledets' theorem, for any $f$-invariant ergodic probability measure $\mu$, there are real numbers $\chi_{\min}^c(\mu) \leq \chi^c_{\max}(\mu) $, called the extremal center Lyapunov exponents, such that for $\mu$-almost every $x$ and every nonzero $v\in E^c(x)$, the forward and backward limits exist and satisfy
\[\chi^c_{\min}(\mu) \leq \lim_{n \to \pm \infty} \frac{1}{n}
\log \bigl\| Df^n(v) \bigr\|  \leq \chi^c_{\max}(\mu) .\]
The following lemma is a special case of a well-known result that allows one to deduce~\emph{uniform growth} from knowledge about exponents for~\emph{every} invariant measure. The proof is a corollary of a classical result on subadditive sequences (see~\cite{sch98} or~\cite[Section~4]{kal11}).  

\begin{lemma}[{see e.g.~\cite[Lemma 10]{dwx21}}]\label{lemma dwx}
If $\chi_{\min}^c(\mu)=\chi_{\max}^c(\mu)=0$ for every $f$-invariant ergodic measure $\mu$, then $Df|_{E^c}$ has subexponential growth, i.e., for every $\varepsilon > 0$ there exists $C_\varepsilon > 0$ such that
\[
\sup_{x\in M}\|D_xf^n|_{E^c(x)}\| \le C_\varepsilon e^{\varepsilon |n|}, \quad \forall n \in \ZZ.
\]
\end{lemma}
Moreover, if $Df|_{E^c}$ has subexponential growth, then $f$ is $\infty$-bunched in the sense that $f$ is $r$-bunched for any $r >0$. We will use the following consequence:

\begin{lemma}[see~\cite{hps77,psw97}]\label{lemma r-bunched}
Let $f:M\to M$ be a $C^\infty$ partially hyperbolic diffeomorphism. Suppose that $f$ is dynamically coherent and that $Df|_{E^c}$ has subexponential growth. Then
    \begin{enumerate}
        \item the leaves of $W^{cs}$, ${W}^{cu}$, and ${W}^c$ are uniformly $C^\infty$; 
        \item the local stable and unstable holonomy maps between $W^c$-leaves are uniformly $C^\infty$. 
    \end{enumerate}
\end{lemma}

\section{Proofs}
\subsection{FAOI diffeomorphisms}
We begin with some basic properties of FAOI (fiberwise Anosov over the identity) diffeomorphisms. 
\begin{lemma}\label{lemma FAOI ph}
Let $f_0:M \to M$ be an FAOI diffeomorphism. Then 
\begin{itemize}
    \item $f_0$ is partially hyperbolic, plaque expansive, and dynamically coherent;
    \item $E^s \oplus E^u$ is integrable to a $C^\infty$ foliation $W^{su}$ with compact leaves.
\end{itemize}
\end{lemma}

\begin{proof}
Let $\pi_0:M\to N$ be the bundle projection. Since $f_0$ is Anosov along the fibers and projects to the identity map on $N$, the graph transform \cite{hps77} therefore gives an invariant complement $E^c$ with
\[
TM=E^s\oplus E^c\oplus E^u,
\qquad \ker D\pi_0=E^s\oplus E^u,
\]
and 
\[
K^{-1}\|v\|\leq\|D_xf_0^n(v)\|\leq K\|v\|,
\qquad \forall x\in M,\ n\in\mathbb Z,\ v\in E^c(x)
\]
for some $K\geq1$. This proves partial hyperbolicity. Dynamical coherence and plaque expansivity follow from \cite[Corollary~7.6]{hhucoherence}. Since the leaves of $W^{su}$ are precisely the fibers of $\pi_0$, they form a $C^\infty$ foliation with compact leaves. 
\end{proof}
The following proposition is the main result in this subsection. 

\begin{proposition}\label{prop single}
Let $f_0$ be an FAOI diffeomorphism and let $f$ be a $C^\infty$ diffeomorphism sufficiently $C^1$-close to $f_0$. Suppose there exists a bi-H\"older homeomorphism $h$ such that 
\[h \circ f_0 = f \circ h.\]
Then $f$ is also FAOI, and $\W^{su} = h(W^{su})$ is a $C^\infty$ foliation with compact leaves.
\end{proposition}
Let $\pi_0:M\to N$ be the bundle projection for $f_0$. Since $f_0$ is FAOI, it is partially hyperbolic, plaque expansive, and dynamically coherent. So any diffeomorphism $f$ sufficiently $C^1$-close to $f_0$ is also partially hyperbolic, plaque expansive, and dynamically coherent  \cite{hps77}. Let 
\[
TM = E^s \oplus E^c \oplus E^u \quad \text{and} \quad TM = \E^s \oplus \E^c \oplus \E^u
\]
be the dominated splittings of $f_0$ and $f$, respectively, and let $W^\ast$ and $\W^\ast$ denote the invariant foliations tangent to $E^\ast$ and $\E^\ast$ for $\ast= s,c,u,cs,cu$. 

\begin{lemma}\label{lemma foliation}
For every $x \in M$ and $\ast = s,c,u,cs,cu$,
\[
h(W^\ast(x)) = \W^\ast(h(x)).
\]
\end{lemma}
\begin{proof}
We first show that all center Lyapunov exponents of $f$ vanish. Suppose that an $f$-invariant ergodic probability measure $\mu$ has a negative center exponent. At a regular point $x$, the local Pesin stable manifold then has dimension strictly larger than $\dim E^s$. Choose a sufficiently small disk $D$ in this manifold, then the exponential contracting property implies that
\[
h^{-1}(D)\subset W^s_{\mathrm{loc}}(x_0),
\]
which contradicts with the dimension assumption. This rules out negative center exponents. Applying the same argument to $f^{-1}$ rules out positive center exponents. By Lemma~\ref{lemma dwx}, for every $\varepsilon>0$, there exists a constant $C_\varepsilon$ such that
\begin{equation}\label{eq subexp perturbation}
\sup_{x\in M}\|D_xf^n|_{\E^c(x)}\|
\leq C_\varepsilon e^{\varepsilon|n|},
\qquad n\in\mathbb Z.
\end{equation}

We next identify the strong foliations. For a partially hyperbolic diffeomorphism with subexponential center growth, exponential convergence of forward orbits characterizes the strong stable leaves. This applies to both $f_0$ and $f$. Since a bi-H\"older conjugacy preserves exponential convergence, it follows that
\[
h(W^s(x))=\W^s(h(x)).
\]
The corresponding argument for the inverse maps gives the assertion for the unstable foliations.

For the center foliation, the uniform bounds on $Df_0^n|_{E^c}$ imply that the iterates of a sufficiently small center plaque remain uniformly small for all $n\in\mathbb Z$. Uniform continuity of $h$ and plaque expansivity of $f$ then places this image in a single local center plaque of $f$. Thus $h$ maps each $W^c$-leaf into a $\W^c$-leaf. By invariance of domain, these images are open in the target leaves; since they partition each connected target leaf, each image is the entire leaf. Finally, the assertions for $cs$ and $cu$ follow from the local product structure of the center and strong foliations.
\end{proof}
Since $E^s \oplus E^u$ is integrable to a foliation $W^{su}$, Lemma~\ref{lemma foliation} implies that $\W^s$ and $\W^u$ are topologically jointly integrable, i.e., there exists a continuous foliation $\W^{su}$ of dimension $\dim \W^s + \dim \W^u$ which is subfoliated by $\W^s$ and $\W^u$.

\begin{lemma}\label{lemma su cinfty}
$\W^{su}$ is a $C^\infty$ foliation.
\end{lemma}

\begin{proof}
Since $\W^{su}$ is subfoliated by $\W^s$ and $\W^u$, both of which have uniformly $C^\infty$ leaves, it follows from \cite[Lemma~4.1]{ks06} that $\W^{su}$ has uniformly $C^\infty$ leaves tangent to $\E^s \oplus \E^u$.

The subexponential estimate \eqref{eq subexp perturbation} and Lemma~\ref{lemma r-bunched} imply that the center leaves and the local stable and unstable holonomy maps between them are uniformly $C^\infty$. By the local product structure inside a $\W^{su}$-leaf, its local holonomy between center plaques is a composition of stable and unstable holonomies. Hence this holonomy is uniformly $C^\infty$ as well. By Journ\'e's lemma \cite{jou88}, $\W^{su}$ is a $C^\infty$ foliation. 
\end{proof}
In particular, the foliation $\W^{su}=h(W^{su})$ gives a smooth fiber bundle $\pi':M\to M/\W^{su}$. The map induced by $h$ on the leaf spaces conjugates the quotient dynamics of $f_0$ and $f$. Since the former is the identity, so is the latter. Moreover, $f$ is uniformly hyperbolic on each fiber because its tangent bundle is $\E^s\oplus\E^u$. Thus $f$ is FAOI with respect to $\pi'$, completing the proof of Proposition~\ref{prop single}.

\subsection{Proof of Theorem~\ref{theorem rank 1}}\label{sec t1}
We first show that $\alpha_0$ and $\alpha$ are topologically conjugate.
\begin{lemma}\label{lemma topo conjugacy}
There exists a bi-H\"older continuous homeomorphism $h:M \to M$ such that
\[
h \circ \alpha_0(\mathbf{n}) = \alpha(\mathbf{n}) \circ h, \quad \forall \mathbf{n} \in \ZZ^k.
\]
\end{lemma}

\begin{proof}
Let $\mathbf{n}_A \in \ZZ^k$ be such that $\alpha_0(\mathbf{n}_A)$ is Anosov. Since $\alpha(\mathbf{n}_A)$ is $C^1$-close to $\alpha_0(\mathbf{n}_A)$, the structural stability of Anosov diffeomorphisms implies that there exists a bi-H\"older continuous homeomorphism $h: M \to M$ that is $C^0$-close to $\mathrm{Id}_M$, such that
\begin{equation}\label{eq commute}
h \circ \alpha_0(\mathbf{n}_A) = \alpha(\mathbf{n}_A) \circ h.
\end{equation}
Define
\[
\rho(\mathbf{n}) := h^{-1} \circ \alpha(\mathbf{n}) \circ h.
\]
Then $\rho$ is a continuous $\ZZ^k$-action satisfying $\rho(\mathbf{n}_A) = \alpha_0(\mathbf{n}_A)$.

Let $\delta>0$ be an expansivity constant for $F_0=\rho(\mathbf{n}_A)=\alpha_0(\mathbf n_A)$, and fix generators $\mathbf e_1,\ldots,\mathbf e_k$ of $\mathbb Z^k$. By taking the perturbation sufficiently small, we may ensure that
\[
d_{C^0}\bigl(\rho(\mathbf e_i)\circ\alpha_0(-\mathbf e_i),\mathrm{Id}_M\bigr)<\delta,
\qquad 1\leq i\leq k.
\]
Each map $g_i=\rho(\mathbf e_i)\circ\alpha_0(-\mathbf e_i)$ commutes with $F_0$. Hence, for every $x\in M$ and $j\in\mathbb Z$,
\[
d\bigl(F_0^j(g_i x),F_0^j x\bigr)
=d\bigl(g_i(F_0^j x),F_0^j x\bigr)<\delta.
\]
Expansivity gives $g_i=\mathrm{Id}_M$. Thus $\rho(\mathbf e_i)=\alpha_0(\mathbf e_i)$ for every generator, and consequently $\rho(\mathbf n)=\alpha_0(\mathbf n)$ for every $\mathbf n\in\mathbb Z^k$.
\end{proof}

We now prove Theorem~\ref{theorem rank 1}.
Let $\pi_0:M\to N$ be the given fiberwise Anosov rank-one factor with a FAOI diffeomorphism $\alpha_0(\mathbf n_f)$. Lemma~\ref{lemma topo conjugacy} supplies a common bi-H\"older conjugacy $h$ for the actions $\alpha_0$ and $\alpha$. Applying Proposition~\ref{prop single} to $\alpha_0(\mathbf n_f)$ and $\alpha(\mathbf n_f)$ gives a smooth fiber bundle $\pi':M\to N'$ whose fibers are precisely the images under $h$ of the fibers of $\pi_0$. In particular, $\pi':M\to N'$ induces a $C^\infty$ rank one factor. Since $\alpha(\mathbf n_f)$ is also FAOI,  $\pi'$ actually induces a fiberwise Anosov rank-one factor for $\alpha$.

\subsection{Proof of Theorem~\ref{theorem product}}
By Lemma~\ref{lemma topo conjugacy}, there exists a bi-H\"older homeomorphism $h$ such that $h \circ \alpha_0 = \alpha \circ h$. Since $\alpha_0(1,0)$ is an FAOI diffeomorphism, it is partially hyperbolic with
\[
W^{su}_1(x,y) = M_1 \times \{y\}, \quad \forall (x,y) \in M_1 \times M_2.
\]
By Proposition~\ref{prop single}, $\alpha(1,0)$ is also FAOI and partially hyperbolic. Moreover, the strong distribution $\E^s_{\alpha(1,0)} \oplus \E^u_{\alpha(1,0)}$ integrates to a compact $C^\infty$ foliation satisfying
\[
\W^{su}_1(h(x,y)) = h\bigl(W^{su}_1(x,y)\bigr), \quad \forall (x,y) \in M_1 \times M_2.
\]
Similarly, by considering $\alpha_0(0,1)$ and $\alpha(0,1)$, we obtain
\[
\W^{su}_2(h(x,y)) = h\bigl(W^{su}_2(x,y)\bigr), \quad \forall (x,y) \in M_1 \times M_2,
\]
where $W^{su}_2(x,y) = \{x\} \times M_2$.

Choose $(x_0,y_0)\in M_1\times M_2$, and put
\[
\L_1=h(M_1\times\{y_0\}),\qquad
\L_2=h(\{x_0\}\times M_2).
\]
These are compact smooth leaves. The two smooth foliations $\W^{su}_1$ and $\W^{su}_2$ are transverse, since their tangent distributions are close to those of the original product foliations. Moreover, every leaf of one meets every leaf of the other in exactly one point, as follows from the topological conjugacy. Therefore
\[
\Psi(p)=\bigl(\W^{su}_2(p)\cap\L_1,\,
                   \W^{su}_1(p)\cap\L_2\bigr)
\]
defines a smooth diffeomorphism $M\to\L_1\times\L_2$. In this coordinate, the map $h$ takes the form
\[
h: M_1 \times M_2 \to \L_1 \times \L_2, \quad (x,y) \mapsto \bigl(h_1(x), h_2(y)\bigr),
\]
and the action $\alpha(\mathbf{n}) = h \circ \alpha_0(\mathbf{n}) \circ h^{-1}$ takes the form
\[
\alpha(\mathbf{n}): \L_1 \times \L_2 \to \L_1 \times \L_2, \quad (x,y) \mapsto \bigl(h_1 \circ f_1^m \circ h_1^{-1}(x),\; h_2 \circ f_2^n \circ h_2^{-1}(y)\bigr).
\]
Although $h_1$ and $h_2$ are in general only H\"older continuous, the maps
\[
f_i' := h_i \circ f_i \circ h_i^{-1} : \L_i \to \L_i, \quad i=1,2,
\]
are $C^\infty$ diffeomorphisms, since they arise as the components of the $C^\infty$ diffeomorphisms $\alpha(1,0)$ and $\alpha(0,1)$. Moreover, $f_i'$ is Anosov because it is the restriction of the corresponding FAOI element to a fiber.

It remains to identify $\L_i$ smoothly with $M_i$. Let $p_i:M_1\times M_2\to M_i$ be the coordinate projection. Since $T\L_i$ is close to the corresponding factor distribution, $p_i|_{\L_i}$ is a local diffeomorphism and, by compactness, a covering map. Since $h$ is $C^0$-close to the identity map,  $\L_i$ is $C^0$-close to $M_i$, the projection $p_i|_{\L_i}$ has one sheet and is thus a diffeomorphism. 

\section{Applying the main result to the Borel-Smale example}\label{ex:borel-smale}
Theorem \ref{theorem rank 1} also applies to Anosov actions which are not product actions, on general nilmanifolds. One example is obtained via the classical Borel--Smale construction \cite{Smale}. By applying Theorem \ref{theorem rank 1}, we are able to show that certain $C^\infty$ fiber bundle persists under sufficiently
$C^1$-small $C^\infty$ perturbations of the action.

Let $\mathfrak g$ be the real Lie algebra with basis $X_1,X_2,Y_1,Y_2,Z_1,Z_2$
whose only nonzero brackets are
\[
[X_1,Y_1]=Z_1,\qquad
[X_1,Y_2]=[X_2,Y_1]=Z_2,\qquad
[X_2,Y_2]=Z_1-Z_2.
\]
Denote
\[
V_X=\operatorname{span}\{X_1,X_2\},\qquad
V_Y=\operatorname{span}\{Y_1,Y_2\},\qquad
V_Z=\operatorname{span}\{Z_1,Z_2\}.
\]
Let $G$ be the simply connected Lie group with Lie algebra $\mathfrak g$. The subgroup
\[
\Gamma=
\bigl\langle
\exp X_1,\exp X_2,\exp Y_1,\exp Y_2,\exp Z_1,\exp Z_2
\bigr\rangle
\]
is a cocompact lattice in $G$. The quotient $M=\Gamma\backslash G$ is a nilmanifold admitting a natural torus fibration
\[
p:M\longrightarrow
\mathbb T^4=\mathbb T^2_X\times\mathbb T^2_Y
\]
induced by the quotient $\mathfrak g\to\mathfrak g/V_Z$ with fiber $\mathbb T^2$. This bundle is nontrivial, since $\pi_1(M)=\Gamma$ is non-abelian.

Consider the hyperbolic matrix
\[
A=\begin{pmatrix}2&1\\1&1\end{pmatrix}.
\]
Relative to $\mathfrak g=V_X\oplus V_Y\oplus V_Z$, define
\[
F=\operatorname{diag}(A,I,A),
\qquad
H=\operatorname{diag}(I,A,A).
\]
It is direct to check that these are commuting Lie algebra automorphisms, and both $F$ and $H$ preserve the brackets. Their induced automorphisms of $G$ preserve $\Gamma$
and therefore descend to commuting diffeomorphisms
$f_0,g_0$ of $M$.

Define the $\mathbb Z^2$-action
\[
\alpha_0(m,n)=f_0^m g_0^n,
\]
then $\alpha_0(1,1)$ is Anosov and $\alpha_0$ is an Anosov action. 

Consider the coordinate projections
\[
\pi_X=\operatorname{pr}_X\circ p,
\qquad
\pi_Y=\operatorname{pr}_Y\circ p.
\]
They satisfy
\[
\pi_X\circ\alpha_0(m,n)=A^m\circ\pi_X,
\qquad
\pi_Y\circ\alpha_0(m,n)=A^n\circ\pi_Y.
\]
Their fibers are four-dimensional tori, tangent to $V_Y\oplus V_Z$ and $V_X\oplus V_Z$ respectively. The map $g_0$ fixes every fiber of $\pi_X$ and acts on its tangent bundle by $A\oplus A$.
Similarly, $f_0$ fixes every fiber of $\pi_Y$ and acts on its tangent bundle by $A\oplus A$.
Consequently, both $\pi_X$ and $\pi_Y$ are fiberwise Anosov rank-one factors.

Let $\alpha$ be a $C^\infty$ $\mathbb Z^2$-action sufficiently $C^1$-close to $\alpha_0$. Applying Theorem~\ref{theorem rank 1} to these two factors gives smooth fiber bundle projections
\[
\pi_{X,\alpha}:M\to N_X,
\qquad
\pi_{Y,\alpha}:M\to N_Y,
\]
where $N_X$ and $N_Y$ are diffeomorphic to $\mathbb T^2$. Moreover, let $\mathcal F_X,\mathcal F_Y$ denote the original
fiber foliations, and let $\mathcal F_{X,\alpha},\mathcal F_{Y,\alpha}$
denote the perturbed ones, then Theorem~\ref{theorem rank 1} gives
\[
\mathcal F_{X,\alpha}=h(\mathcal F_X),
\qquad
\mathcal F_{Y,\alpha}=h(\mathcal F_Y).
\]
In particular, 
\[
p_\alpha
=
(\pi_{X,\alpha},\pi_{Y,\alpha})
:M\longrightarrow N_X\times N_Y
\]
is a $C^\infty$ fiber bundle over $N_X\times N_Y\cong\mathbb T^4$. Therefore, in this example, Theorem \ref{theorem rank 1} also implies that the  $\mathbb T^2$-fibration (corresponding initially to the center of the Lie algebra) persists
as a $C^\infty$ foliation under sufficiently
$C^1$-small $C^\infty$ perturbations of the action.

\bibliographystyle{amsalpha}
\bibliography{ref}
\end{document}